\documentclass[a4paper,10pt,reqno]{amsart}
\usepackage{amssymb,amscd,a4,textcomp, latexsym,psfrag,amsmath, calc, multirow, color, cancel, graphicx, verbatim, epsfig, graphicx}
\usepackage[english]{babel}
\usepackage[utf8]{inputenc}
\usepackage{subfigure}
\usepackage[all, cmtip]{xy}
\usepackage{diagrams}
\usepackage[normalem]{ulem}

\usepackage[usenames,dvipsnames]{pstricks}
\usepackage{epsfig}
\usepackage{pst-grad} 
\usepackage{pst-plot} 

\newtheorem{lemma}{Lemma}[section]
\newtheorem{teo}[lemma]{Theorem}

\newtheorem{defz}[lemma]{Definition}
\newtheorem{prop}[lemma]{Proposition}

\title[Rational Picard group of moduli of higher spin curves]{ON THE RATIONAL PICARD GROUP OF THE MODULI SPACE OF HIGHER SPIN CURVES}

\begin{document}

\maketitle

\begin{center}
  \textsc{Letizia Pernigotti} 
 \end{center}

\vspace{1cm}
\begin{center}\parbox{0.8\textwidth}{
\begin{small}\textsc{Abstract} -- We refine the notion of higher spin curves defined in terms of line bundles, by adding  
the additional structure of coherent nets of roots firstly introduced by Jarvis in terms of torsion-free sheaves. Next, we describe the boundary
part of their moduli space in the case without marked points. Finally, we provide a presentation of the rational Picard group of this space.

\vspace{0,4cm}


\vspace{0.2cm}

\textsc{2010 MR Subject Classification}: 14H10, 14C22, 14C20. \end{small}
}\end{center}

%


\vspace{1cm}

\section{Introduction}

In \cite{Jarvis2001}, the moduli space of higher spin curves with the additional data of \emph{coherent net of roots} is described 
using torsion-free sheaves.
The aim of this note is to refine the notion of root given
in the line bundles setting by \cite{Caporaso2007} with such additional data and investigate the geometry of the corresponding moduli space. 
This space is less singular than the one described in \cite{Caporaso2007} and it is thus possible to study its Picard group.

A stable $r$-spin curve of type $\mathbf{m}$ is an $n$-pointed curve of genus $g$ with only ordinary nodes as 
singularities together with the data of a coherent net of $r$-th roots of type
$\mathbf{m}$ on it.
The set $\overline{S}^{1/r,\mathbf{m}}_{g,n}$ of stable $r$-spin curves up to isomorphisms is in one-to-one correspondence with the moduli space
$\overline{\mathcal{S}}_{g,n}^{1/r,\mathbf{m}}$ of higher spin curves
described by \cite{Jarvis2001}, hence we can transport the scheme structure of the 
latter to the 
former and consider the two spaces as isomorphic schemes.

\begin{prop}\label{propfirst}
There is a scheme with underlying set $\overline{S}^{1/r,\mathbf{m}}_{g,n}$ (which we will still 
call $\overline{S}^{1/r,\mathbf{m}}_{g,n}$)
isomorphic to the scheme $\overline{\mathcal{S}}^{1/r,\mathbf{m}}_{g,n}$ introduced by Jarvis in 
\cite{Jarvis2001}. In particular, $\overline{S}^{1/r,\mathbf{m}}_{g,n}$ is normal and projective.
\end{prop}

In order to describe the rational Picard group of this moduli space, we focus on the case without marked points and
we restrict to one irreducible component whenever $S^{1/r}_g:=S^{1/r,\mathbf{0}}_{g,0}$ is not irreducible. 
The low dimensional homology and cohomology of the open part $S_{g}^{1/r}$ has been studied in \cite{Randal-Williams2012},
together with the relations between boundary divisors and other
classes, such as the classes $\lambda$ and $\mu$. It turns out that, when the genus $g$ is greater or equal to $9$, the first holomology 
group over $\mathbb{Q}$ is zero, while the second cohomology group has rank $1$. This means that, for instance, the Hodge class
generates the Picard group of $S_{g}^{1/r}$ over $\mathbb{Q}$.

By \cite{Jarvis2001}, the classes
$\lambda$, $\{\alpha_i^{(a,b)}\}$ and $\{\gamma_{j,\eta}^{(a,b)}\}$, where the $\alpha_i^{(a,b)}$'s and $\gamma_{j,\eta}^{(a,b)}$'s denote suitable boundary
divisors (to be defined in Section 3), 
are independent in $\mathrm{Pic}(\overline{S}^{1/r}_g)$. By \cite{Jarvis2000} and \cite{Jarvis2001} the space 
$\overline{S}^{1/r}_g$ is normal and with quotient singularities, 
so that, as in \cite{BiniGilberto;Fontanari2006}, its rational Picard group is isomorphic to the correspondent Chow group over $\mathbb{Q}$. This 
implies that the whole Picard group $\mathrm{Pic}(\overline{S}^{1/r}_g)$ is generated over $\mathbb{Q}$ by the generators of the Chow group of the open part $S^{1/r}_g$ together
with the set of boundary classes of $\overline{S}^{1/r}_g$. 
In the end we obtain a complete description of the generators of the rational Picard group.

\begin{teo}
 Assume $g\geq 9$. Then $\mathrm{Pic}(\overline{S}^{1/r}_g)$ is freely generated over $\mathbb{Q}$ by the classes $\lambda$, $\{\alpha_i^{(a,b)}\}$ and 
 $\{\gamma_{j,\eta}^{(a,b)}\}$, where $\lambda$ is the Hodge class and $\{\alpha_i^{(a,b)}\}$ and 
 $\{\gamma_{j,\eta}^{(a,b)}\}$ are the boundary divisors.
\end{teo}

We work over the complex field $\mathbb{C}$.

\section{The moduli space}\label{sect:basic}
 
 \subsection{Coherent nets}\label{subsect1}
 Le $C$ be a stable curve of genus $g$. Let us start by recalling the two basic definitions of $r$-th root of a line bundle and of a torsion-free sheaf.
\begin{defz}[\cite{Jarvis2001}, $r$-th root of a torsion-free sheaf]
 Given a semistable curve $C$ and a rank-one torsion-free sheaf $\mathcal{K}$ on $C$, an \emph{$r$-th root of $\mathcal{K}$} is a
 pair $(\mathcal{E},b)$ of a rank-one torsion-free sheaf $\mathcal{E}$ and an $\mathcal{O}_C$-module homomorphism
 $b:\mathcal{E}^{\otimes r}\to\mathcal{K}$ such that
 \begin{enumerate}
  \item $r\cdot \deg \mathcal{E}=\deg\mathcal{K}$,
  \item $b$ is an isomorphism on the locus of $C$ where $\mathcal{E}$ is not singular,
  \item for every $p\in C$ where $\mathcal{E}$ is singular, it is $\mathrm{length}_p(\mathrm{coker}(b))=r-1$.
 \end{enumerate}
\end{defz}

\begin{defz}[\cite{Caporaso2007}, limit $r$-th root of a line bundle]
 Given a stable curve $C$, a line bundle $K\in\mathrm{Pic}(C)$ and an integer $r$ dividing $\deg(K)$, a \emph{limit $r$-th root} 
 of $(C,K)$ is a triple $(X,L,\alpha)$ where
 $\pi:X\to C$ is a blow-up of $C$ at a set of nodes $\Delta$, $L\in\mathrm{Pic}(X)$ and $\alpha$ is a
 homomorphism $L^{\otimes r}\to\pi^*K$ satisfying
 \begin{enumerate}
  \item the restriction of $L$ to every exceptional component has degree one,
  \item the map $\alpha$ is an isomorphism outside the exceptional components,
  \item for every exceptional components $E_i$ of $X$, the orders of vanishing of $\alpha$ at $p_i^+,p^-_i\in E_i$ add up to $r$,
  where $\{p^+_i,p_i^-\}=\pi^{-1}(p)$.
 \end{enumerate}
\end{defz}

Since the moduli space of stable curves with $r$-th roots of a fixed sheaf is not smooth when $r$ is not prime, Jarvis rigidified
the space through the additional structure of \emph{coherent net of roots} (see \cite{Jarvis2001}). We aim to translate this construction
in terms of line bundles on semi-stable curves.

\begin{defz}[\cite{Jarvis2001}, coherent net of roots]
 Given a semistable curve $C$ and a rank-one torsion-free sheaf $\mathcal{K}$ on $C$, a \emph{coherent net of roots} for $\mathcal{K}$
 is a collection $\{\mathcal{E}_d,c_{d,d'}\}$ consisting of a rank-one torsion-free sheaf $\mathcal{E}_d$ for every $d$ dividing $r$ and a
 homomorphism $c_{d,d'}:\mathcal{E}_d^{\otimes d/d'}\to \mathcal{E}_{d'}$ for each $d'$ dividing $d$ such that
 \begin{itemize}
  \item $\mathcal{E}_1=\mathcal{K}$ and $c_{d,d}=\mathrm{Id}$ for each $d$ dividing $r$,
  \item for every $d'|d|r$ the pair $(\mathcal{E}_d,c_{d,d'})$ is a $d/d'$-th root of $\mathcal{E}_{d'}$ in such a way that all the maps
  are compatible. In other words it is $c_{d',d''}\circ c_{d,d'}^{\otimes d'/d''}=c_{d,d''}$ for every $d''|d'|d|r$.
 \end{itemize}
\end{defz}

If $r$ is prime, then this costruction reduces to a simply $r$-th root of $\mathcal{K}$. Moreover, if for some $d$ the sheaf $\mathcal{E}_d$ is locally free, then 
it determines up
to isomorphism all pairs $(\mathcal{E}_{d'},c_{d,d'})$ such that $d'|d$ by the relations $\mathcal{E}_{d'}=\mathcal{E}_d^{\otimes d/d'}$, $c_{d',1}=c_{d,1}$ and $c_{d,d'}=\mathrm{Id}$.
This means that the whole net is not already encoded in the ``top'' root $(\mathcal{E}_r,c_{r,1})$ only when, for some $d<r$, at a node $p$ the sheaf $\mathcal{E}_d$ is locally free but $\mathcal{E}_r$ is not. In this
case the net corresponds to the choice of a locally free $d$-th root $\mathcal{E}_d$ of $\mathcal{K}$ and a non-locally-free $(r/d)$-th root of $\mathcal{E}_d$ (\cite[\S 2.3.2]{Jarvis2000}). 

Let us go more into details.
Consider a stable curve $C$ and a coherent net of roots $\{\mathcal{E}_d,c_{d,d'}\}$ of $\omega_C$. Let $(\mathcal{E}_r,c_{r,1})$ be the top root and let $\Delta$ be
the set of points of $C$ where $\mathcal{E}_r$ is singular. Let us call $\nu: Y\to C$ the partial normalization of $C$ at $\Delta$.
For simplicity suppose $\Delta=\{p\}$ and let $\{u,v\}$ be the order of the $r$-th root $(\mathcal{E}_r,c_{r,1})$ at $p$. This 
means (see \cite[\S 2.2.2]{Jarvis2001}) that
\[
\nu^\natural\mathcal{E}_r^{\otimes r}\stackrel{\tilde{c}_{r,1}}{\cong}\nu^*\omega_C(-up^+-vp^-)=\omega_Y(-(u-1)p^+-(v-1)p^-),
\]
where $\nu^\natural$ is defined as $\nu^\natural \mathcal{F}:=\nu^*\mathcal{F}/tors$ for every sheaf $\mathcal{F}$, the map $\widetilde{c}_{r,1}$ is the 
isomorphism induced by the map $c_{r,1}$ at the level of the partial normalization $Y$ and the points $\{p^+,p^-\}$ are the preimages of the point $p$ under
the map $\nu$. Moreover, for every $d|r$, the order of the $d$-th root $(\mathcal{E}_d,c_{d,1})$ at $p$ is given by $\{u_d,v_d\}$, where $u_d$ and $v_d$
are the least non-negative integers congruent respectively to $u$ and $v$ modulo $d$ (see 	\cite{Jarvis2001}). From this it is clear that $\mathcal{E}_d$ is smooth at $p$ if and only if
$d$ divides $u$ (and $v$). We can thus distinguish between the case in which $u$ and $v$ are relatively prime and the case in which they are not. Let us define
\[
\nu^\natural\mathcal{E}_d:=\nu^\natural\mathcal{E}_r^{\otimes r/d}\otimes \mathcal{O}_Y\left(\frac{u-u_d}{d}p^++\frac{v-v_d}{d}p^-\right).
\]

If $\mathrm{gcd}(u,v)=1$ then no $c_{d,1}$ is an isomorphism, all $\mathcal{E}_d$'s are singular at $p$ and furthermore they are completely determined by the top root by the relations
$\mathcal{E}_d=\nu_*\nu^\natural\mathcal{E}_d$ and $c_{d,1}=c_{r,1}$ (see \cite[\S 2.2.2]{Jarvis2001}).

If instead $\mathrm{gcd}(u,v)=l>1$ then the sheaf $\mathcal{E}_l$ is smooth at $p$ and it is necessary an additional gluing datum to construct $\mathcal{E}_d$ from $\nu^\natural\mathcal{E}_l$.
In particular, the sheaf $\mathcal{E}'_l:=\nu^\natural\mathcal{E}_l$ is a smooth (that is, of order $\{0,0\}$) $l$-th root of $\omega_Y(p^++p^-)$ and thus the $l$-th root $\mathcal{E}_l$ is
recovered once we choose an isomorphism $\phi: {\mathcal{E}'_l}_{|p^+}\to {\mathcal{E}'_l}_{|p^+}$ compatible with the maps
\[
{\mathcal{E}_l^{'\otimes l}}_{|p^+} \stackrel{\sim}{\longrightarrow} \omega_Y(p^++p^-)\stackrel{\sim}{\longleftarrow} {\mathcal{E}_l^{'\otimes l}}_{|p^-}.
\]
In other words, the gluing datum corresponds to a (non-canonically determined) $l$-th root of unity and hence there are exactly $l$ distinct possibilities for the gluing. We will
see later that in some cases different gluings correspond to automorphisms of the underlying curve while in other cases different gluings produce different spin curves.

\subsection{Torsion-free sheaves and line bundles}\label{subsect2}
We restate here the main theorem linking torsion-free sheaves on stable curves and line bundles on semi-stable curves.
\begin{prop}[\cite{Caporaso2007}, Proposition 4.2.2.]\label{teo:bundshea}
 Let $B$ an integral scheme and $f:\mathcal{C}\to B$ and $f:\mathcal{C}\to B$ a family of nodal curves.
 \begin{enumerate}
  \item Let $\pi:\mathcal{X}\to\mathcal{C}$ be a family of blow-ups of $\mathcal{C}$ and let $\mathcal{L}\in\mathrm{Pic}\mathcal{X}$ be a line bundle 
  having degree one on every exceptional component. Then $\pi_*(\mathcal{L})$ is a relatively torsion-free sheaf of rank one, flat over $B$.
  \item Conversely, suppose that $\mathcal{F}$ is a relatively torsion-free sheaf of rank one on $\mathcal{C}$, flat over $B$. Then there exist a family
   $\pi:\mathcal{X}\to \mathcal{C}$ of blow-ups of $\mathcal{C}$ an a line bundle $\mathcal{L}\in\mathrm{Pic}\mathcal{X}$ having degree one on all exceptional 
   components, such that $\mathcal{F}\cong \pi_*(\mathcal{L})$.
  \item Let $\pi:\mathcal{X}\to \mathcal{C}$ and $\pi':\mathcal{X}'\to \mathcal{C}$ be families of blow-ups of $\mathcal{C}$ and
  $\mathcal{L}\in\mathrm{Pic}\mathcal{X}$, $\mathcal{L}'\in\mathrm{Pic}\mathcal{X}'$ line bundles having degree one on every exceptional component. Then
  $ \pi_*(\mathcal{L})\cong \pi'_*(\mathcal{L}') \iff \exists \sigma:\mathcal{X}\stackrel{\sim}{\longrightarrow}\mathcal{X}' \text{ isomorphism over }
  \mathcal{C} \text{ s.t. } \mathcal{L}\cong\sigma^*(\mathcal{L}'). $
 \end{enumerate}
\end{prop}

Consider a coherent net of roots $\{\mathcal{E}_d,c_{d,d'}\}$ for $(C,\omega_C)$. By Proposition \ref{teo:bundshea}, we can associate 
to every sheaf $\mathcal{E}_d$ a blow-up $\pi_d:X_d\to C$ of $C$ at the points where 
$\mathcal{E}_d$ is singular and a line bundle $L_d\in\mathrm{Pic}(X_d)$ having degree one on all exceptional components and such that
$\mathcal{E}_d\cong {\pi_d}_*(L_d)$.
Furthermore, for every $d$ dividing $r$, the map $c_{d,1}:\mathcal{E}_d^{\otimes d}\to\omega_C$ corresponds to an
homomorphism $\alpha_d:L_d^{\otimes d}\to \pi_d^*\omega_C$ that makes $(X_d,L_d,\alpha_d)$ a limit $d$-th
root of $(C,\omega_C)$ in the sense of \cite{Caporaso2007}. In particular, if $\{u_d,v_d\}$ is the order of
$(\mathcal{E}_d,c_{d,1})$ at a point $p\in C$ where $\mathcal{E}_d$ is singular, 
then it is also the order of vanishing of
$\alpha_d$ at the points $\{p^+,p^-\}=\pi_d^{-1}(p)$ of $X_d$.

Suppose that there exists two integers $d$ and $d'$ such that at a point $p\in C$ the sheaf $\mathcal{E}_d$ is singular while 
the sheaf $\mathcal{E}_{d'}$ is smooth.
\begin{figure}[h]
   \centering
   \includegraphics[width=0.75\textwidth]{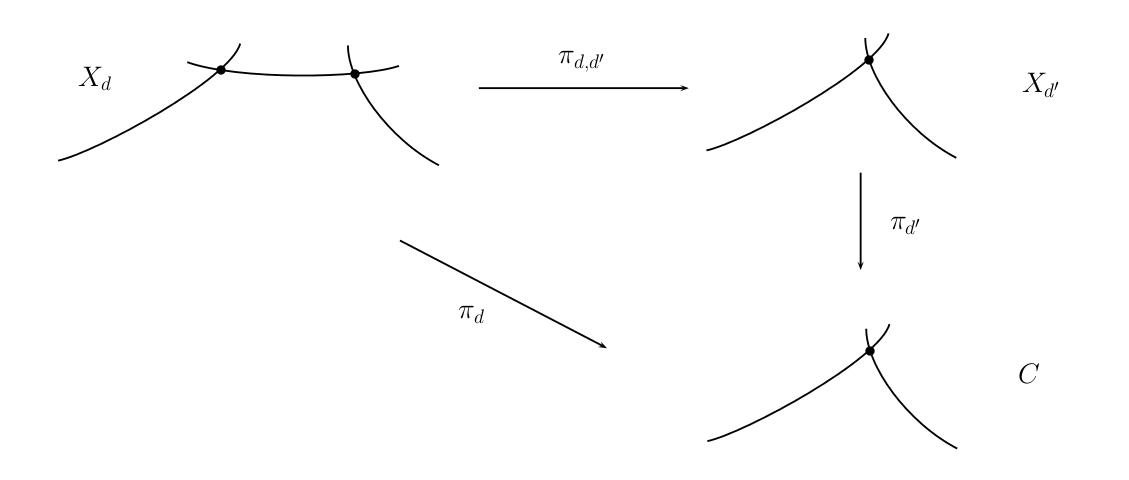}
   \caption{$u$ and $v$ not relatively prime.}
   \label{fig:nrelprim}
\end{figure}

As illustrated in Figure \ref{fig:nrelprim}, in this case $X_d\neq X_{d'}$ and we have a sequence of blow-ups
\[
\xymatrix{   X_d \ar[r]^{\pi_{d,d'}} \ar@/_2pc/[rr]_{\pi_d} & X_{d'} \ar[r]^{\pi_{d'}} & C }
\]
and two line bundles $L_d\in \mathrm{Pic}(X_d)$ and $L_{d'}\in\mathrm{Pic}(X_{d'})$ on two different bases. But if $\{p^+,p^-\}$ 
are the two points in the preimage $\pi_{d,d'}^{-1}(p)=\pi_d^{-1}(p)$, we know that the line bundle $L_{d'}$ is completely
determined by the line bundle $\widetilde{N}_{d'}:=\pi_{d,d'}^*(L_d)$ plus a gluing datum 
$\eta_{d'}^p:\widetilde{N}_{d'|p^+}\stackrel{\sim}{\to}\widetilde{N}_{d'|p^-}$.
In other words, giving a line bundle $N\in \mathrm{Pic}(X_{d'})$ is the
same (modulo isomorphisms) as giving a line bundle $\widetilde{N}\in\mathrm{Pic}(X_d)$ together with an isomorphism 
$\eta^p:{\widetilde{N}}_{|p^+}\stackrel{\sim}{\to}{\widetilde{N}}_{|p^-}$
such that $\widetilde{N}_{X_d\setminus E}=L_{X_{d'}\setminus\{p\}}$ and $\widetilde{N}_{|E}\cong \mathcal{O}_E(1)$, where $E=\pi_{d,d'}^{-1}(p)$ is the exceptional 
component coming from the blow-up.

Consider again the coherent net of roots $\{\mathcal{E}_d,c_{d,d'}\}$ and let $r$ be the order of the top root. Define $X:=X_r$ and $\pi:=\pi_r$
and let $\Delta:=\{p_1,\ldots,p_k\}\subset C$ be the set of points where $\mathcal{E}_r$ is singular. As we have seen above, to the
coherent net of roots corresponds the collection of limit roots $\{(X_d,L_d,\alpha_d)\}$ for every $d$ dividing $r$. Since we prefer
to have a fixed base space $X$ for all line bundles, we replace the collection $\{(X_d,L_d,\alpha_d)\}$ with the collection
$(X,\{N_d,\beta_d,\eta^{p_j}_d\})$ where $N_d:=\pi_{r,d}^*L_d\in\mathrm{Pic}(X)$, the map $\beta_d$ is the homomorphism $N_d^{\otimes d}\to \pi^*\omega_C$ turning
$N_d$ into a limit $d$-th root of $(C,\omega_C)$ and $\eta^{p_j}_d$ is an isomorphism ${N_d}_{|p_j^+}\stackrel{\sim}{\to} {N_d}_{|p_j^+}$ every
time that $\beta_d$ has order zero at $\{p_j^+,p_j^-\}$.

We may note that, if $N:=N_r$ is the top root and $\nu:Y\to C$ is the partial normalization at $\Delta$, then the line bundle $N$
is exactly the gluing of a degree one line bundle on each exceptional component with a line bundle $\widetilde{N}\in \mathrm{Pic}(Y)$ satisfying
\begin{equation}\label{eq:toproot}
\widetilde{N}^{\otimes r}=\nu^*\omega_C\left(-\sum_j(u^jp_j^++v^jp_j^-)\right),
\end{equation}
where $\{u^j,v^j\}$ are the order of vanishing of $\beta:=\beta_r$ at the points $p_j^+$ and $p_j^-$ (see \cite[\S 2.2]{Caporaso2007}).
Furthermore, the map $\beta$ is defined to agree with the inclusion $\nu^*\omega_C\left(-\sum_j(u^jp_j^++v^jp_j^-)\right)\hookrightarrow \nu^*\omega_C$ on
$Y$ and to be zero on the exceptional components.
Any other line bundle of the collection is then obtained by gluing the line bundle
\begin{equation}\label{eq:othroot}
\widetilde{N}_d:=\widetilde{N}^{\otimes r/d}\otimes \mathcal{O}_Y\left(\sum_j \frac{u^j-u_d^j}{d}p^++\frac{v^j-v_d^j}{d}p^-\right)
\end{equation}
on $Y$ with the degree one line bundle on the exceptional components plus the datum of a gluing $\eta_d^{p_j}$ every time is needed, that is
every time that $d$ divides $u$ (or $v$). 

We are now ready to state the main definition. 
\begin{defz}[Coherent net of roots]\label{def:netlb}
Given a stable curve $C$ and an integer $r$ dividing $2g-2$, a \emph{coherent net of roots} of order $r$ for $(C,\omega_C)$ is the datum 
$(X,\{N_d,\beta_{d,d'},\eta^{p_j}_d\}_{d|r})$ where
\begin{itemize}
 \item $\pi:X\to C$ is a blow-up of $C$ at a set of nodes $\Delta=\{p_1,\ldots,p_k\}$,
 \item $N_1=\pi^*\omega_C$ and $\beta_{d,d}=\mathrm{Id}$ for every $d$,
 \item for every $d|r$, the triple $(X,N_d,\beta_{d,1})$ is a limit $d$-th root of $(C,\omega_C)$ and the map $\beta_{d,1}$ has
 order $\{u_d^j,v_d^j\}$ at the points $p^+_j$ and $p_j^-$ for every $j\in\{1,\ldots,k\}$,
 \item for every $d'|d$, the map $\beta_{d,d'}:N_d^{\otimes d/d'}\to N_d$ is an isomorphism outside the exceptional components and has order
 $\{(u^j_d-u^j_{d'})/d', (v^j_d-v^j_{d'})/d'\}$ at the points $p^+_j$ and $p_j^-$ for every $j\in\{1,\ldots,k\}$,
 \item for every $j\in\{1,\ldots,k\}$, if $d$ divides $u$ then $\eta^{p_j}_d:{N_d}_{|p_j^+}\stackrel{\sim}{\to} {N_d}_{|p_j^-}$ is a gluing datum, otherwise it is the null map.
 \item all the maps $\beta_{d,d'}$ are compatible.
\end{itemize}
\end{defz}

As in \cite{Jarvis2001} we can define what an isomorphism of coherent net of $r$-th roots is.
\begin{defz}[Isomorphism of coherent nets]
 Given a stable curve $C$, an \emph{isomorphism of coherent nets of $r$-th roots} of $(C,\omega_C)$ 
 from $(X,\{N_d,\beta_{d,d'},\eta^{p_j}_d\}_{d|r})$ to
 $(X',\{N'_d,\beta'_{d,d'},\eta'^{p_j}_d\}_{d|r})$ is an isomorphism of curves $\tau:X\to X'$
together with a system of isomorphisms $\{\gamma_d:\tau^*N'_d\to N_d\}$ compatible with all the maps $\tau^*\beta'_{d,d'}$, $\beta_{d,d'}$,
$\tau^*\eta'^{p_j}_d$ and $\eta^{p_j}_d$
and such that $\gamma_1$ is the canonical isomorphism
$ \beta_1:\tau^*\pi'^*\omega_C\to \pi^*\omega_C. $
\end{defz}

Define the set
\begin{equation*}\label{eqsch}
\overline{S}^{1/r}_g:=\left\{\Big[C,(X,\{N_d,\beta_{d,d'},\eta^{p_j}_d\}_{d|r})\Big] \left|
\begin{array}{c}
C \text{ is a stable curve of genus }g, \\
(X,\{N_d,\beta_{d,d'},\eta^{p_j}_d\}_{d|r}) \text{ is a coherent net of} \\
r\text{-th root in the sense of Def. \ref{def:netlb}}
\end{array}\right.\right\},
\end{equation*}
where $[\cdot,\cdot]$ denotes the class of $r$-spin curves modulo isomorphisms.

\begin{prop}\label{prop1}
There is a scheme with underlying set $\overline{S}^{1/r}_g$ (which we will still call $\overline{S}^{1/r}_g$)
isomorphic to the scheme $\overline{\mathcal{S}}^{1/r}_g$ introduced by Jarvis in 
\cite{Jarvis2001}. In particular, $\overline{S}^{1/r}_g$ is normal and projective.
\end{prop}
\begin{proof}
 By construction, there is a set-theoretical bijection between $\overline{S}^{1/r}_g$ and the coarse moduli space $\overline{\mathcal{S}}^{1/r}_g$
 of $\overline{\mathfrak{S}}^{1/r}_g$. Thanks to this correspondence, we can 
 transport the scheme structure of the latter to the former and consider the two spaces as isomorphic schemes. Since the latter
 is normal and projective (see \cite[Proposition 3.1.1]{Jarvis2001}), the proof is over. 
\end{proof}

\subsection{Marked points}
The whole construction can be repeated considering also curves with marked points. We state here the analogous definitions.
Let $C$ be an $n$-pointed curve of genus $g$ with only ordinary nodes as singularities,
$\mathbf{m}=(m^1,\ldots,m^n)$ an $n$-tuple of integers and $r\in \mathbb{N}$ a positive
integer dividing $2g-2-\sum m^i$. Following \cite{Caporaso2007}, we define  a \textit{$r$-th-root} as follows.

\begin{defz}[Limit $r$-th root of type $\mathbf{m}$]
 Let $(C,(q_1,\ldots,q_n))$ be an $n$-pointed curve with only ordinary nodes as singularities.
 We say that $((X,(q_1,\ldots,q_n)),L,\alpha)$ is an $r$-th root of $((C,(q_1,\ldots,q_n)),\omega_C)$ of type
$\mathrm{\mathbf{m}}$ if $(X,(h_1,\ldots,h_n))$ is a blow-up of $(C,(q_1,\ldots,q_n))$ at a
set of nodes $\Delta=\{p_1,\ldots,p_k\}$, $L$ is a line bundle on $X$ and $\alpha$ is a homomorphism 
$L^{\otimes r}\to \pi^*\omega_C(-\sum m^iq_i)$ satisfying
 \begin{enumerate}
 \item the restriction of $L$ to every exceptional component has degree one,
 \item the map $\alpha$ is an isomorphism outside the exceptional components
 \item for every exceptional component $E_j$, the order of vanishing $u^j$ and $v^j$ of $\alpha$ at 
$\{p_j^+,p_j^-\}=\pi^{-1}(p_j)$ add up to $r$,
\end{enumerate}
\end{defz}

The arguments of \S\ref{subsect1} and \S\ref{subsect2} still hold with little modifications. In particular, equation (\ref{eq:toproot}) becomes
\begin{equation}\label{eq:toproot2}
\widetilde{N}^{\otimes r}=\nu^*\omega_C\left(-\sum_j(u^jp_j^++v^jp_j^-)-\sum_i m^iq_i\right),
\end{equation}
while equation (\ref{eq:othroot}) becomes
\begin{equation}\label{eq:othroot2}
\widetilde{N}_d:=\widetilde{N}^{\otimes r/d}\otimes \mathcal{O}_Y\left(\sum_j \frac{u^j-u_d^j}{d}p^++\sum_j\frac{v^j-v_d^j}{d}p^-
+\sum_i \frac{m^i-m_d^i}{d}q_i\right).
\end{equation}
The definition of coherent net of roots is completely analogous to the one without marked points.
The only additional requirement is that each $N_d$ is now a $d$-th root of $(C,\omega_C)$ of type $\mathbf{m}_d:=(m^1_d,\ldots,m^n_d)$ and that
the order of the maps $\beta_{d,d'}$ is consequently modified.
Since the results of \cite{Jarvis2001} used in Proposition \ref{prop1} hold for $\overline{\mathfrak{S}}^{1/r,\mathbf{m}}_{g,n}$ and
$\overline{\mathcal{S}}^{1/r,\mathbf{m}}_{g,n}$, Proposition \ref{propfirst} follows.

\section{Boundary divisors}

In order to describe the boundary divisors we restrict to the case without marked points.
Since the two schemes $\overline{S}^{1/r}_{g}$ and $\overline{\mathcal{S}}^{1/r}_{g}$ are isomorphic, 
the description of the boundary in $\overline{S}^{1/r}_{g}$ is
straightforward from the one of $\overline{\mathcal{S}}^{1/r}_{g}$ made by Jarvis.

\subsection{Two irreducible components and one node}
Let $C$ be a stable curve with two irreducible components $C_1$ and $C_2$ of genus respectively $i$ and $g-i$ meeting in a double
point $p$. Let $\pi:X\to C$ be its blow-up at $p$ and $\nu:Y\to C$ the normalization. Given a coherent net of roots for $(C,\omega_C)$, let us
call the top root $(X,N_r,\beta_r)$ and let $\{u,v\}$ be the order of $\beta_{r,1}$ at $\{p^+,p^-\}$.
We have already seen that outside the exceptional component it is
\[
\widetilde{N}_r^{\otimes r}=\nu^*\omega_C(-up^+-vp^-)=\omega_{Y}(-(u-1)p^+-(v-1)p^-).
\]
We must have the degree of $\nu^*\omega_C(-up^+-vp^-)$ divisible by $r$ and this implies there exists a unique choice possible for $u$ 
(and $v=r-u$), given by
\[
u\equiv \deg\omega_{C_1}=2i-1 \ \mod r, \qquad v\equiv \deg \omega_{C_2}=2g-2i-1 \ \mod r.
\]
See for instance \cite[p. 29]{Caporaso2007}.

If $u\equiv 0$ mod $r$, then $X=C$ and $N_r\in \mathrm{Pic}(C)$ corresponds to a line bundle $N_{1,r}\in\mathrm{Pic}(C_1)$ satisfying $N^{\otimes r}_{1,r}\cong\omega_{C_1}(p^+)$
and a line bundle $N_{2,r}\in \mathrm{Pic}(C_2)$ satisfying $N^{\otimes r}_{2,r}\cong\omega_{C_2}(p^-)$ plus the datum
of a gluing datum $\eta_r^p:{N_{1,r}}_{|p^+} \stackrel{\sim}{\to}{N_{2,r}}_{|p^-}$. Any other line bundle of the net is then determined
by $N_d:=N_r^{\otimes r/d}$, $\beta_{d,d'}=\mathrm{Id}$ and $\eta_d^p:=\eta_d^{p\otimes r/d}$. However, different choices for the gluing datum 
$\eta_r^p$ correspond to automorphisms of the $r$-spin structure of $C_1$ or $C_2$ and hence induces different but isomorphic $r$-spin structure on
$C$. In other words, different gluings correspond to the same point in $\overline{S}^{1/r}_g$. See for instance \cite[\S 1.7.1]{JKV01}.

If $u\not\equiv 0$ mod $r$, we may distinguish two cases. If $\mathrm{gcd}(u,v)=1$, then there is no need of gluing datum and,
as said before, all the remaining of the net is completely determined by the top root by equation (\ref{eq:othroot}).
If $\mathrm{gcd}(u,v)=l>1$ then there is the need to specify a gluing datum $\eta_p^d:{N_d}_{|p^+}\stackrel{\sim}{\to}{N_d}_{|p^-}$ but,
as in the case $u\equiv 0$, all $d$ gluing data will yield non-canonically isomorphic $N_d$ and hence isomorphic nets of roots.

In conclusion, a spin curve with two irreducible components and one node corresponds to an element of 
$\mathfrak{S}^{1/r,u-1}_{i,1}\times\mathfrak{S}^{1/r,v-1}_{g-i,1}$ for $u$ and $v$ adding up to $r$
(see \cite[Example 1]{Jarvis2001}).

\subsection{One irreducible component and one node}
Let $C$ be a stable curve with one node $p$, let $X$ be its blow-up in $p$ and $Y$ its normalization. Given a coherent net or $r$-th
root as before, this time there are $r$ possible choice for $u$ and $v$ that allow a spin structure on $X$: or $u=v=0$, in which
case $X=C$ and the map $\beta_{r,1}$ is an isomorphism, or $u\in\{1,\ldots,r-1\}$ and $v=r-u$, in which case the map
$\beta_{r,1}$ has order exactly $\{u,v\}$ at $\{p^+,p^-\}$.

If $\mathrm{gcd}(u,v)=1$, then the spin structure on $C$ is completely determined by the $r$-th root $(X,N_r,\beta_{r,1})$ which in turn 
is induced by an $r$-th root of $\omega_Y(-(u-1)p^+-(v-1)p^-)$, where $Y$ is this time connected and irreducible. In particular,
this kind of spin curve corresponds to an element of $\mathfrak{S}^{1/r,(u-1,v-1)}_{g-1,2}$ and, furthermore, there is a morphism
$\mathfrak{S}^{1/r,(u-1,v-1)}_{g-1,2}\to \overline{\mathfrak{S}}^{1/r}_g$ whose image contains both points with 
$\widetilde{N}_r^{\otimes r}=\omega_Y(-(u-1)p^+-(v-1)p^-)$ and $\widetilde{N}_r^{\otimes r}=\omega_Y(-(v-1)p^+-(u-1)p^-)$.

If $\mathrm{gcd}(u,v)=l>1$ or $u=v=0$, then the coherent net requires the additional datum of a gluing isomorphism
$\eta_l^p$ (or $\eta_r^p$ in the second case). This time, an automorphism of the $r$-spin structure on $Y$ induces
the same automorphism on both sides of the isomorphism $\eta_l^p$ and hence it preserves the gluing. Consequently, we have
$l$ distinct gluing morphism (or $r$ distinct gluing morphims in the second case) corresponding to $l/2$ (rounding up) different points
in $\overline{S}^{1/r}_g$ belonging to $l/2$ distinct irreducible boundary divisors (see \cite[Example 2 and \S 3.2.2.]{Jarvis2001}), since
the two points in the normalization are not distinguishable.

\quad\\

Recall that, for $g>1$, the moduli spaces $\overline{S}^{1/r,\mathbf{m}}_{g,n}$ are irreducible if $r$ is odd and they are the disjoint union
of two irreducible component if $\mathrm{gcd}(r,m^1,\ldots,m^n)$ is even.
The two components corresponds respectively to $r$-spin curves with $h^0(N_2)$ even
 and $h^0(N_2)$ odd. When $g=1$, the moduli space $\overline{S}^{1/r,\mathbf{m}}_{1,n}$ is the disjoint union of $\rho$ irreducible components, where
 $\rho$ is the number of divisors of $\mathrm{gcd}(r,m^1,\ldots,m^n)$ (see \cite[\S 1.3]{JKV01} and \cite[Theorem 2.7]{Jarvis2001}).
 
We can now summarize the description of the boundary divisors of $\overline{S}^{1/r}_g$ keeping the same notation 
of \cite{Jarvis2001}. We have seen that, over a stable curve in $\delta_i\subset\overline{M}_g$, given
the order $r$ of the root, there is a unique possibility for the order of the top root. This means that over $\delta_i$
there are the irreducible divisors $\alpha_i^{(a,b)}$ corresponding to the locus of spin curves
with a spin structure of index $a$ on the genus $i$ component and of index $b$ on the genus $g-i$ component. Here the
index $a$ parametrizes the components of $S^{1/r,u-1}_{i,1}$ while the index $b$ parametrizes
the components of $S^{1/r,u-1}_{i,1}$ (see \cite[\S 3.2.2.]{Jarvis2001}).

Over $\delta_0$ there are different types of spin structures. For any choice of the order $\{u,v\}$, when
$u$ and $v$ are not relatively prime the spin structure
is determined except for the gluing. Also in this situation, for a particular choice of order $\{u,v\}$, of index $a$ of the component and of 
the gluing $\eta$, the corresponding divisor of spin curves of the given order, index and gluing is irreducible (see \cite[\S 3.2.2]{Jarvis2001}).
We denote these divisors by $\gamma_{j,\eta}^{(a)}$, where $j:=\min(u,v)$, $\eta$ is the gluing datum and $a$ is the index of the correspondent
component
of $S^{1/r,(u-1,v-1)}_{g-1,2}$. Since the two points in the normalization are not distinguishable, there is the additional
relation $\gamma_{j,\eta}^{(a)}=\gamma_{r-j,\eta^{-1}}^{(a)}$.

\begin{teo}\label{teo2}
 Assume $g\geq 9$. Then $\mathrm{Pic}(\overline{S}^{1/r}_g)$ is freely generated over $\mathbb{Q}$ by the classes $\lambda$, $\{\alpha_i^{(a,b)}\}$ and 
 $\{\gamma_{j,\eta}^{(a,b)}\}$ just described.
\end{teo}
\begin{proof}
 By \cite[Theorem 2.3]{Jarvis2000}, the space $\overline{\mathfrak{S}^{1/r}_g}$ is a smooth proper Deligne-Mumford stack over 
 $\mathbb{Z}[1/r]$ and by \cite[Theorem 2.7]{Jarvis2001} its coarse moduli space is normal and projective. 
 By \cite[Proposition 2.8]{Vistoli1989}, if a scheme of finite type over a field of characteristic zero is the
moduli space of some smooth stack, then its normalization has quotient singularities. 
Since the coarse moduli space of $\overline{\mathfrak{S}^{1/r}_g}$ is itself normal, we can conclude that it has quotient singularities and so also
$\overline{S}^{1/r}_g$ has quotient singularities. This implies that every Weil divisor is a $\mathbb{Q}$-Cartier divisor and hence there is an isomorphism
\[
\mathrm{Pic}(\overline{S}^{1/r}_g)\otimes \mathbb{Q}\cong A_{3g-4}(\overline{S}^{1/r}_g)\otimes \mathbb{Q}.
\]
In particular, as in \cite{BiniGilberto;Fontanari2006}, thanks to the exact sequence
\[
A_{3g-4}(\overline{S}^{1/r}_g\setminus S^{1/r}_g) \to A_{3g-4}(\overline{S}^{1/r}_g)\to A_{3g-4}(S^{1/r}_g)\to 0
\]
we know that the group $\mathrm{Pic}(\overline{S}^{1/r}_g)\otimes \mathbb{Q}$ is generated by the generators of
$ A_{3g-4}(S^{1/r}_g)$ together with the set of boundary classes of $\overline{S}^{1/r}_g$.

Consider first the open part and restrict to one irreducible component whenever $S^{1/r}_g$ is not irreducible. 
To do this, denote by $S^{1/r}_g[\epsilon]$ the whole $S^{1/r}_g\otimes_{\mathbb{Z}[1/r]}\mathbb{C}$ if $r$ is odd
or the component of Arf invariant $\epsilon$ if $r$ is even. 
Theorem 1.4 of \cite{Randal-Williams2012} shows that for $g\geq 6$ it is
\[
H_1(S^{1/r}_g[\epsilon];\mathbb{Q})=0
\]
and for $g\geq 9$ the second cohomology group  $H^2(S^{1/r}_g[\epsilon];\mathbb{Q})$ has rank 1 and it is thus
 generated by only one class. Since the Hodge class $\lambda$ is a non trivial class in this group,
 we can conclude that the open part $\mathrm{Pic}(S^{1/r}_g[\epsilon])\otimes \mathbb{Q}$ is generated, for instance, by $\lambda$.
 
Now we take care of the contributions coming from the boundary part. 
By \cite{Jarvis2001} the $\{\alpha_i^{(a,b)}\}$ and the 
$\{\gamma_{j,n}^{(a,b)}\}$ are 
generators for the boundary divisors of the space $\overline{S}^{1/r}_g$.
Moreover, by Proposition 3.4 of \cite{Jarvis2001}, the classes
$\lambda$, $\alpha_i^{(a,b)}$ and $\gamma_{j,n}^{(a,b)}$ are independent in 
$\mathrm{Pic}(\overline{\mathcal{S}}^{1/r}_g)$ for $g>1$, hence they are independent also in $\mathrm{Pic}(\overline{S}^{1/r}_g)$ and the claim follows.
\end{proof}

\quad\\
\begin{footnotesize}\textsc{Università degli Studi di Trento, Dipartimento di Matematica, \\ Via Sommarive 14,  I-38123 Povo (TN)}\end{footnotesize}

\noindent E-mail address: \verb=pernigotti@science.unitn.it=

 \end{document}